\documentclass[11pt]{amsart}
\usepackage{amsmath,amssymb}
\usepackage{todonotes,tabu,multicol}
\usepackage{graphicx,import}
\usepackage[utf8]{inputenc}
\usepackage{tabularx}
\usepackage{caption}
\usepackage{svg}
\usepackage{fourier}
\usepackage{wrapfig}
\usepackage[colorlinks = true, citecolor = blue, linkcolor = blue]{hyperref}
\usepackage{amsmath,amsthm,amsfonts,amssymb,amscd}
\usepackage{lastpage}
\usepackage{enumerate}
\usepackage{fancyhdr}
\usepackage{mathrsfs}
\usepackage{xcolor}
\usepackage{graphicx}
\usepackage{listings}
\usepackage{hyperref}
\usepackage{amssymb} 
\usepackage{amsmath} 
\usepackage{xfrac}
\usepackage{enumitem}
\usepackage{mathtools}

\usepackage[utf8]{inputenc}
\usepackage{mathtools}
\usepackage [english]{babel}
\usepackage [autostyle, english = american]{csquotes}
\usepackage[margin = 1in]{geometry}
\usepackage[normalem]{ulem}
\usepackage{float}

\newtheorem{theorem}{Theorem}[section]
\newtheorem{lemma}[theorem]{Lemma}
\newtheorem{proposition}[theorem]{Proposition}
\newtheorem{corollary}[theorem]{Corollary}
\newtheorem{conjecture}[theorem]{Conjecture}

\theoremstyle{definition}
\newtheorem{definition}[theorem]{Definition}
\newtheorem{example}[theorem]{Example}

\theoremstyle{remark}
\newtheorem{remark}[theorem]{Remark}

\numberwithin{equation}{section}

\begin{document}

\title{A determinant formula of the Jones Polynomial for a family of braids}
\author{Derya Asaner, Sanjay Kumar, Melody Molander, Andrew Pease, Anup Poudel}
\date{}

\maketitle

\begin{abstract}

In 2012, Cohen, Dasbach, and Russell presented an algorithm to construct a weighted adjacency matrix for a given knot diagram. In the case of pretzel knots, it is shown that after evaluation, the determinant of the matrix recovers the Jones polynomial. Although the Jones polynomial is known to be $\#P$-hard by Jaeger, Vertigan, and Welsh, this presents a class of knots for which the Jones polynomial can be computed in polynomial time by using the determinant. In this paper, we extend these results by recovering the Jones polynomial as the determinant of a weighted adjacency matrix for certain subfamilies of the braid group. Lastly, we compute the Kauffman polynomial of $(2,q)$ torus knots in polynomial time using the balanced overlaid Tait graphs. This is the first known example of generalizing the methodology of Cohen to a class of quantum invariants which cannot be derived from the HOMFLYPT polynomial.

    
\end{abstract}

\section{Introduction}\label{introduction}
A fundamental question of knot theory is determining whether two knots are distinct. A useful tool which has been constructed to answer such a question is the use of knot invariants, such as knot polynomials. In this process, knots are assigned a polynomial, and if they are distinct for two knots, then it is known that these two knots must be different. One such example of a knot polynomial is the famous Jones polynomial.

Although knot polynomials are utile, they are very difficult to compute. Thus, knot theorists have sought quicker ways to compute them. A combinatorial approach through the use of graph theory has been of particular interest. One such method, introduced by Morwen Thistlethwaite \cite{Thi86}, consists of assigning a knot a corresponding graph, called a Tait graph. Then, by performing contractions and deletions of the Tait graph, we can compute the Tutte polynomial associated to the graph. Equivalently, the Tutte polynomial can be obtained by summing over all of the Tait graph's spanning trees with specified weights. Performing contractions and deletions on the Tait graph corresponds to smoothing resolutions of a knot, which arise within the bracket polynomial. In particular, the Tutte polynomial can be specialized to recover the bracket polynomial. 
\par
Cohen, Dasbach, and Russell \cite{Coh14} used Thistlethwaite's work to introduce a similar method that recovers the bracket polynomial for the class of pretzel knots. Specifically, they assign a knot a corresponding graph called a balanced overlaid Tait graph. Then, the edges of the graph are assigned special weightings. By taking the determinant of a specific submatrix of the graph's adjacency matrix, with each entry being the weight of the edge connecting the incident vertices, the bracket polynomial is obtained. Similarly, it is known \cite{Coh12} that the determinant can also be computed directly by considering all perfect matchings of the balanced overlaid Tait graph in the following way. We take the product of all of the weights associated with each specific perfect matching, then sum over all such terms from perfect matchings. The resulting sum will be equivalent to the computed determinant. In the case when the determinant recovers the bracket polynomial, we say that the knot admits a dimer model. As the Jones polynomial is closely related to the bracket polynomial, this method gives a determinant formula for the Jones polynomial for pretzel knots. 
\par
It is well known that the computational complexity of the Jones polynomial of a knot is $\#P$-hard due to the work of Jaeger, Vertigan, and Welsh \cite{Jae92}. However, an advantage in the work of Cohen, Dasbach, and Russell is that they construct an algorithm that can compute the Jones polynomial for a class of knots in polynomial time. Although it is not possible for all of the Jones polynomials to be computed in this manner due to the $\#P$-hard restriction, there may exist other examples where the computationally-fast algorithm applies. One advantage of having computationally-fast methods are its application in conjectures between the geometry and quantum topology of manifolds. For example, the volume conjecture, first posed by Kashaev in \cite{Kas97} and reformulated in terms of the colored Jones polynomial by Murakami and Murakami \cite{MM01}, relies on determining asymptotics of the complex calculations involved in the colored Jones polynomial. Therefore, a method in simplifying the calculations may provide useful in such conjectures.

\par
In this paper, we apply Cohen's results to a different class of knots: homogeneous closed braids of the form $\sigma_1^{m_1}\sigma_2^{m_2}\dots\sigma_{n-1}^{m_{n-1}}$, where the exponents are all negative or positive. Namely, we show in our main theorem that, for this class of knots, the polynomial obtained from the balanced overlaid Tait graph by summing over all the perfect matchings is equivalent to the polynomial obtained from the Tait graph by summing over all the spanning trees. Since Thistlethwaite proved that the spanning tree method is known to give the bracket polynomial for all knots, this will give that the perfect matching method also gives the bracket polynomial for this particular class of knots. As a result, the homogeneous closed braids of the form $\sigma_1^{m_1}\sigma_2^{m_2}\dots\sigma_{n-1}^{m_{n-1}}$ have a determinant formula for their Jones polynomial and is computationally fast. 

Apart from the well known generalization of the Jones and the Alexander polynomial given by the HOMFLYPT polynomial, there are other quantum invariants presented using skein relations that are not obtained as a specialization of the HOMFLYPT polynomial. In this paper we also provide a polynomial-time algorithm to compute one such invariant, the Kauffman polynomial, using the matrices obtained from the balanced overlaid Tait graphs of $(2,q)$-torus knots.  



Section \ref{knottheory} covers the basics of knot theory and definition of the braid group on $n$-strands. The skein theoretic presentation of the Jones and the bracket polynomial are given in Section \ref{knotinvariants}. Graph theoretic preliminaries that lead to the definition of the signed Tait graph and the balanced overlaid Tait graph along with the algorithm to assignment of the activity letters leading to the definition of a dimer model are discussed in Section \ref{graphs}. We work out the details of a dimer model for the $(2,q)$ torus knots in Section \ref{torusdimer} and prove the existence of a dimer model for a family of links obtained as homogeneous braid closures in Section \ref{hombraiddimer}. Section \ref{Kauffman Polynomials} provides a recursive formula for computing the Kauffman polynomial of $(2,q)$ torus knots using the balanced overlaid Tait graphs. Finally, the future directions are discussed in Section \ref{furtherdirections}.


\section{Acknowledgements}
This project was supported by the National Science Foundation Grant DMS-1850663 through the UC Santa Barbara Math REU which was led by Maribel Bueno. 


\section{Knot Theory}\label{knottheory}
A \textbf{\emph{knot}} is a smooth embedding of a circle in $\mathbb{R}^3$. Similarly, a smooth embedding of a disjoint collection of circles into $\mathbb{R}^3$ gives us a \textbf{\emph{link}}. A \textbf{\emph{link diagram}} is the 2-dimensional projection of a link from $\mathbb{R}^3$ onto the plane with crossing information. A \textbf{\emph{crossing}} is a point of intersection within a link. Links may consist of over- and under-crossings, which are illustrated in Figure \ref{trefoilfig}. 

\begin{center}
\begin{figure}[h]
\includegraphics[width = 20mm]{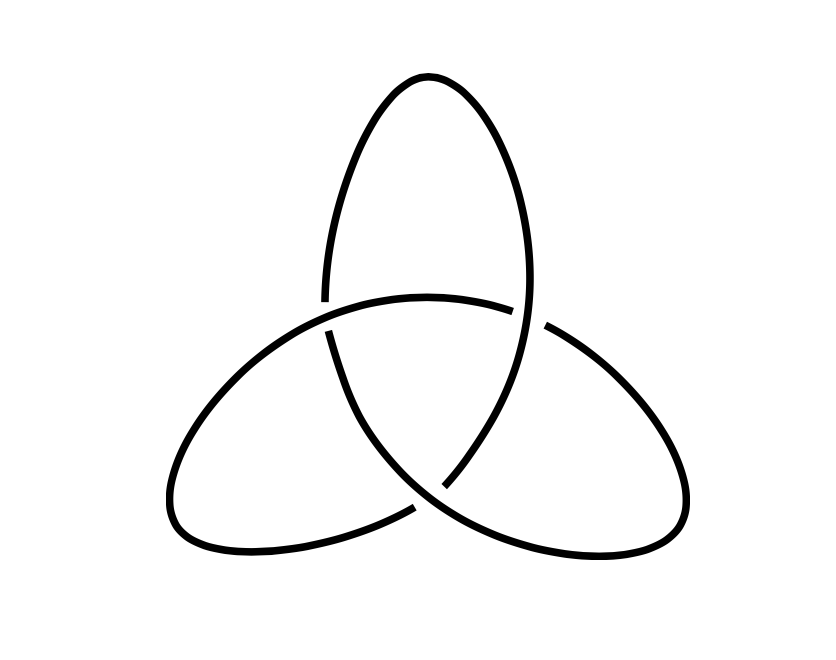}
\caption{diagram of the right-handed trefoil knot.}
\label{trefoilfig}
\vspace{-1cm}
\end{figure}
\end{center}

Two knots (and links) are equivalent if they are isotopic to each other. In terms of knot (and link) diagrams, a classical theorem of Reidemeister \cite{Rei27} tells us that this is equivalent to checking if any two link diagrams differ from one another by a sequence of Reidemeister moves and planar isotopies. 
Inducing an orientation on these crossings allows us to distinguish one from another as shown below.

\begin{center}
\begin{figure}[hbt!]
  \includegraphics[width = 45mm]{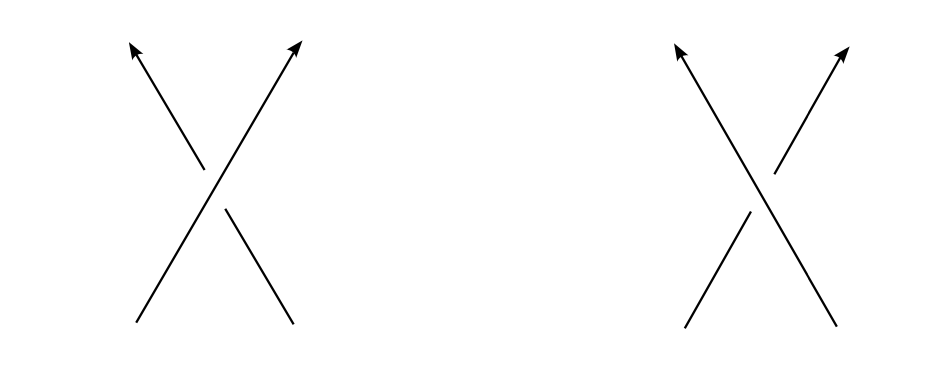}
    \caption{a positive and negative crossing, respectively}
\end{figure}
\vspace{-1cm}
\end{center}


There exist many different classes of links. We are particularly interested in links that can be obtained as closures of homogeneous braids, which will be discussed below. Closed braids arise from the braid group. The \textbf{$n$-strand braid group}, $B_n$, consists of $n$ strands such that the operation is the concatenation of strands. The identity element of $B_n$ consists of $n$ straight strands which do not intersect. Thus, each strand $i$ begins in the $i^{th}$ position on the top and ends at the same $i^{th}$ position at the bottom. 

\begin{center}
\begin{figure}[h]
\includegraphics[width = 80mm]{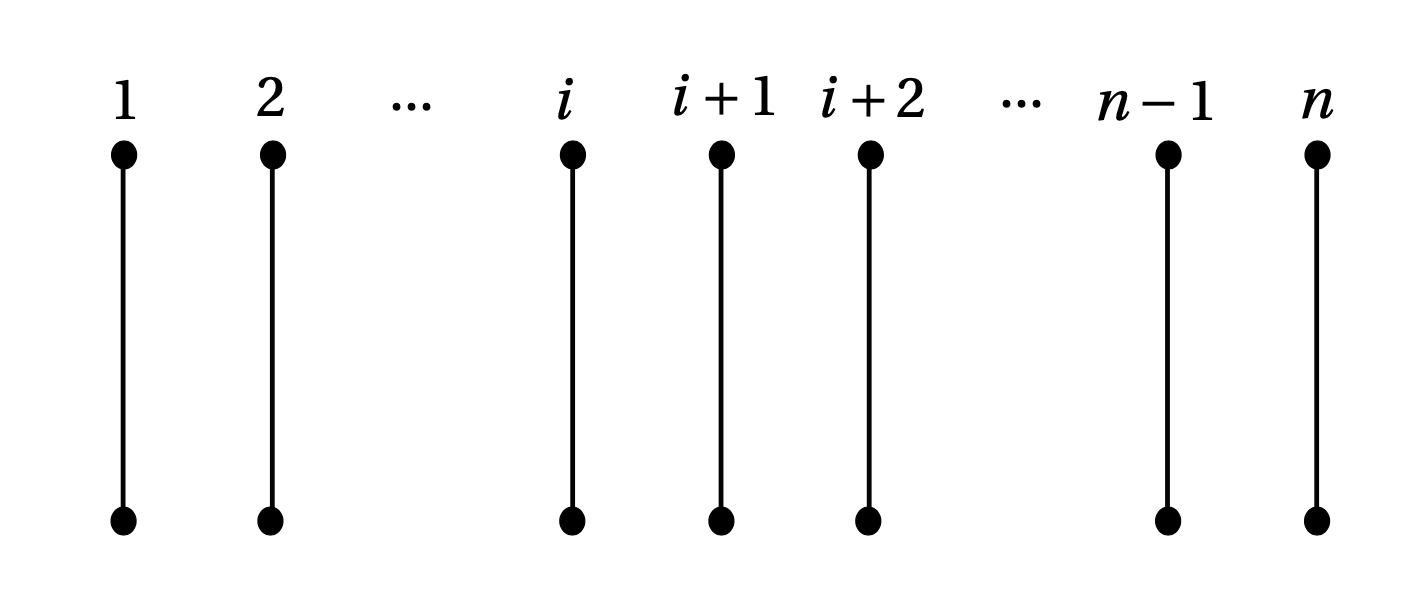}
\caption{$e \in B_n$}
\end{figure}
\end{center} 
\vspace{-1cm}

\noindent A presentation of the braid group on $n$-strands, $B_n$, first defined by Artin \cite{Artin47}, is as follows:
    \begin{equation*}
       B_n= \langle \sigma_1, \dots, \sigma_{n-1} | \sigma_i \sigma_j= \sigma_i \sigma_j \text{ for $|i-j| \geq 2$},\hspace{2mm} \sigma_i \sigma_{i+1} \sigma_i= \sigma_{i+1} \sigma_i \sigma_{i+1} \text{ for  $1 \leq i \leq n-2$} \rangle.
    \end{equation*}
Diagrammatically, we can represent each $\sigma_i$ as the $n$-braid where the $i^{th}$ strand is crossed over the $(i+1)^{th}$ strand.

\begin{center}
\begin{figure}[H]
    \includegraphics[width = 80mm]{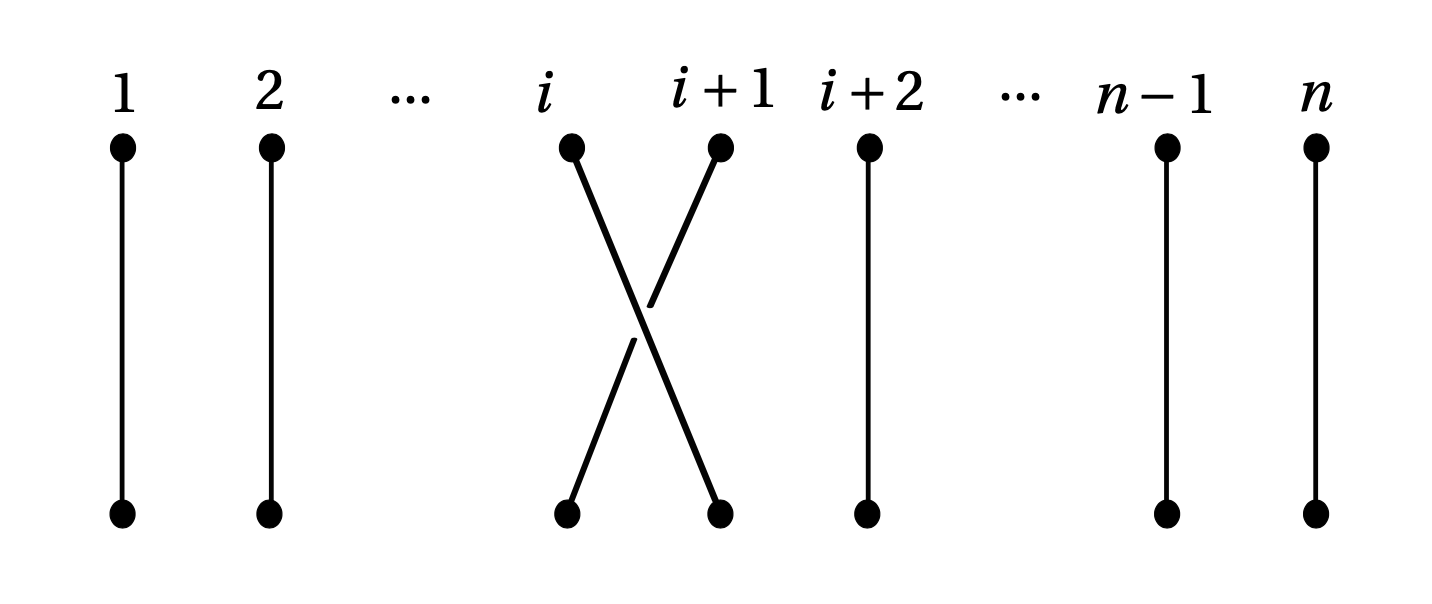}
    \caption{$\sigma_i \in B_n$}\label{braidclosure1}
    \vspace{-1cm}
    \end{figure}
\end{center}

\noindent The inverse of each generator, $\sigma_i^{-1}$,  consists of crossing the $(i+1)^{th}$ strand over the $i^{th}$ strand. A \textbf{\emph{braid word}} consists of multiple generators, which are read from left to right. We take the first generator and perform it individually, then we take the following generator and stack it underneath. To close a braid, connect each strand from the top, to its same position on the bottom on the right. 

\begin{center}
\begin{figure}[H]
\vspace{-1cm}
\includegraphics[scale=.5, width=65mm]{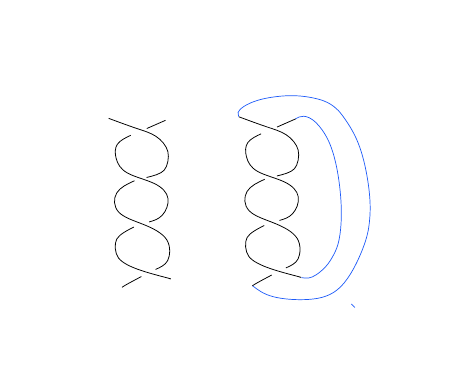}
\vspace{-1cm}\caption{the braid $\sigma_1^{4} \in B_2$ and its closure}
\end{figure}
\vspace{-1cm}
\end{center}

\begin{theorem}[Alexander's Theorem \cite{Ale23}]\label{alexandersthm}
Every knot or link can be represented as some closed braid.
\end{theorem}
\begin{center}
\begin{figure}[h]
\includegraphics[width = 70mm]{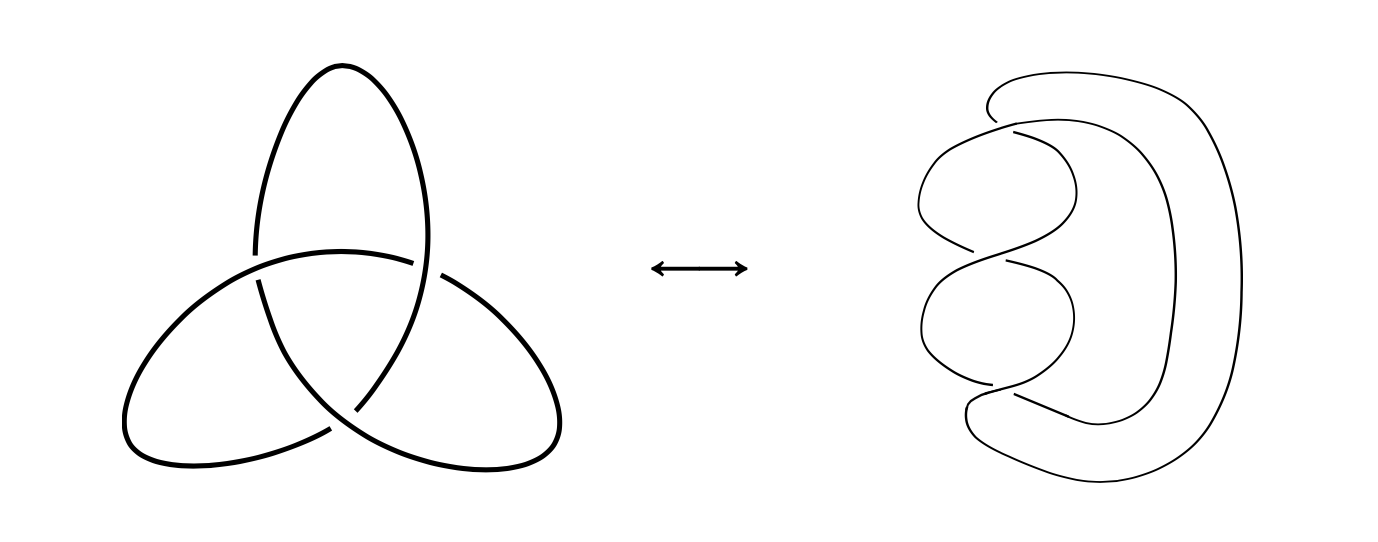}
\caption{the trefoil knot diagram and its equivalent closed braid. }
\end{figure}
\end{center}
\smallskip


\section{Knot Invariants}\label{knotinvariants}
Distinguishing knots from one another (up to isotopy) is a fundamental problem in knot theory. One approach to such a problem is to assign knots an isotopy invariant polynomial, which we call a \textbf{\emph{knot polynomial}}. These polynomials are typically defined in terms of skein relations. These relations are recursions that ``smooth" the knot by resolving or removing crossings.

\begin{center}
\begin{figure}[H]
\includegraphics[width = 75mm]{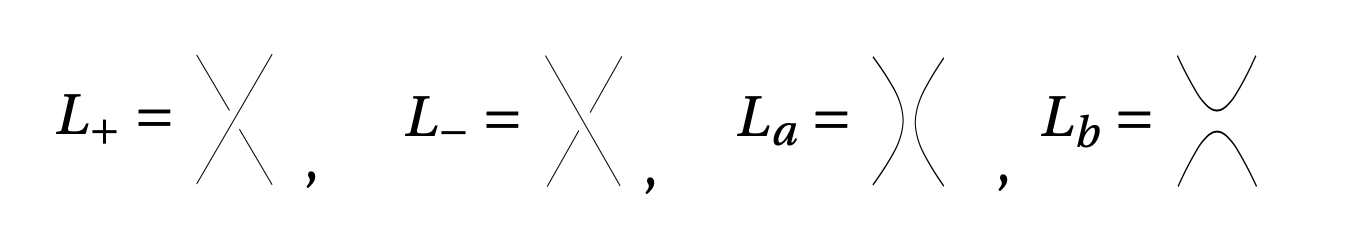}\\
\caption{positive crossing, negative crossing, and a $\&$ b smoothings, respectively}\label{crossingsdef}
\vspace{-1cm}
\end{figure}
\end{center}

For example, the \textbf{\emph{bracket polynomial}} for a link L, denoted $\langle L \rangle$, introduced by Kauffman \cite{Kau87}, obeys the following relations, where $\textrm{O}$ denotes the trivial knot without crossings (the \textit{unknot}): 
           \[\langle \textrm{O} \rangle = 1 \]
          \[\langle L_{+} \rangle = A \langle L_{a} \rangle + A^{-1} \langle L_{b} \rangle\]
         \[\langle \textrm{O} \sqcup L \rangle = (-A^2 - A^{-2}) \langle L \rangle \hspace{3mm} \text{\cite{Kau87}}.\] 
Note that the bracket polynomial is not an invariant because it does not discern the first Reidemeister move. For example, it cannot differentiate between an unknot and an unknot with a twist. 

Jones' groundbreaking research in subfactor theory led to the \textbf{Jones polynomial} \cite{Jon85}, which is defined in terms of the bracket polynomial and is an invariant. To define the Jones polynomial, we must first consider a link's writhe. Each positive and negative crossing in a link diagram is given a value of $+1$ and $-1$, respectively. Then, we define the \textbf{\emph{writhe}} as the sum of the crossings of the link, considering their signs. 
Thus, the Jones polynomial of a link is: \[J(L)=(-A^{-3})^{wr(L)} \langle L \rangle \hspace{3mm}\text{\cite{Jon85}}. \]
The term $(-A^{-3})^{wr(L)}$ is a correction term which resolves this deficiency in the bracket polynomial. 


\section{Graphs}\label{graphs}
A \textbf{\emph{graph}} is a collection of vertices that are connected by edges (or we can call edges \textbf{\emph{dimers}}). A graph comes with the data of a \textbf{\emph{vertex set}} $V$, and an \textbf{\emph{edge set}} $E$. Thus, we define a graph $G$ as tuple $G = (V,E)$. A \textbf{\emph{planar graph}} is a graph that can be topologically embedded into the plane. Not all graphs can be embedded into the plane, though proving this is non-trivial \cite{Kur30}. The \textbf{\emph{faces}} of a planar graph are the regions bounded by edges in the graph. Conventionally, we take the unbounded region on the outside of the graph to also be a face. The \textbf{\emph{dual}} $G^*$ of a graph $G$ has a vertex for every face of $G$, and edges for faces separated by an edge.\par
\par 
There are deep connections between knots, graphs, and these polynomial invariants. We will first consider those between knots and graphs. 

\begin{definition}[Signed crossings for checkerboard shaded links]
    For a link $L$ given a checkerboard coloring, we assign crossings to be either positive or negative using the following convention:
\end{definition}
\begin{center}
\includegraphics[scale=0.2]{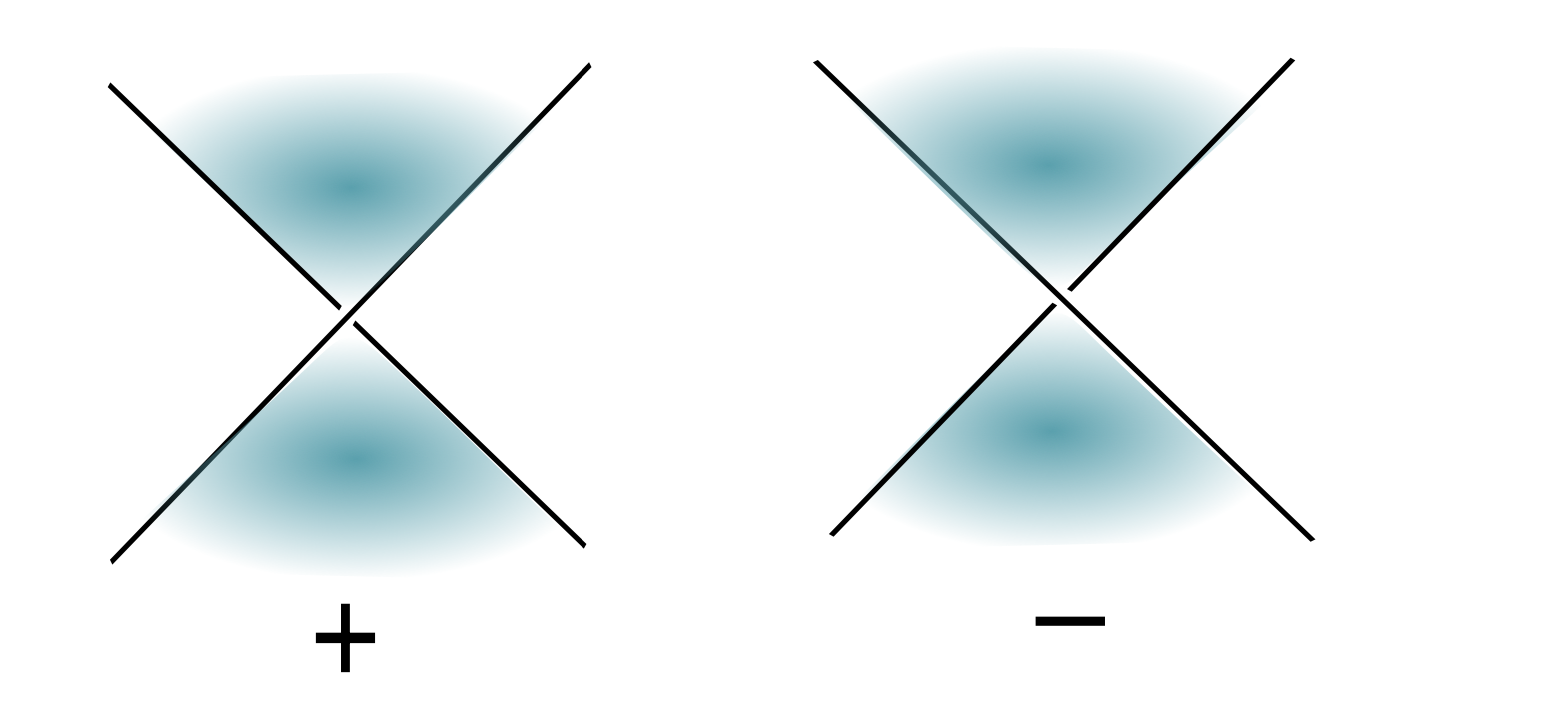}
\end{center}

\begin{definition}[Signed Tait graph \cite{Thi86}] \label{signedtait}
Let $D$ be a checkerboard colored link diagram. A \emph{Tait graph} $G = (V,E)$ has vertex set $V= \{\textrm{colored faces}\}$ and an edge between vertices $v,w\in V$ per face touched along a crossing. There is a choice of checkerboard coloring; the other choice gives us a Tait graph equal to $G^*$.  We sign the edges according to the sign of the crossing giving the \textit{signed Tait graphs}. For the rest of this paper choose the signed Tait graph to be the graph with the outside face not shaded and thus its dual will have the outside face shaded.
\end{definition}

\begin{example}\label{signedtaitgraphexample}
    Below is the construction of the signed Tait graph for the right-handed trefoil which is the closure of the braid $\sigma_1^{3}$:
\end{example}

\vspace{-0.3cm}
\begin{figure}[h!]
    \begin{center}
    \vspace{-0.5cm}
    \includegraphics[width = 60mm]{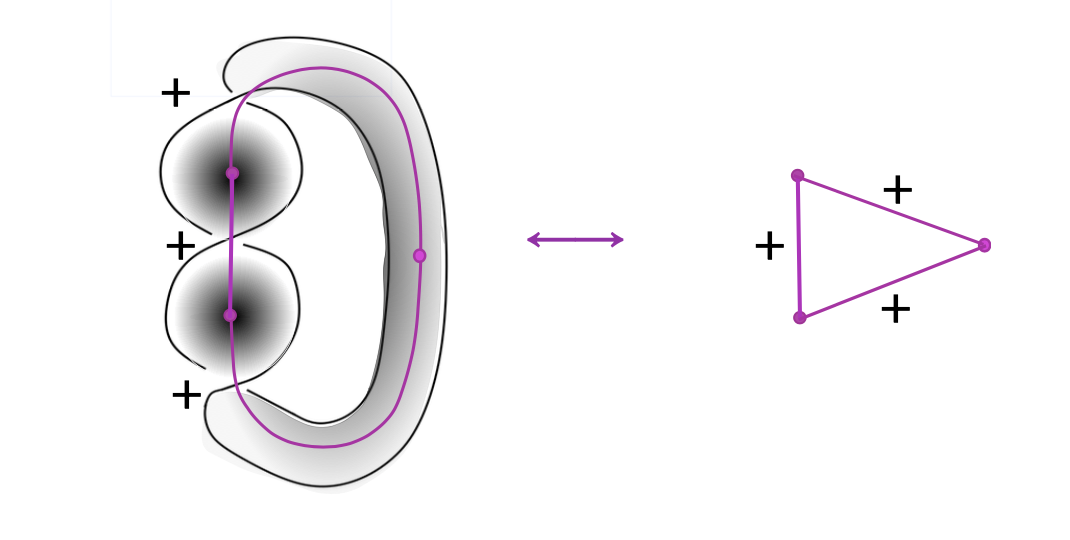}
   \caption{signed Tait graph of the trefoil.}
    \end{center}
\end{figure}

\vspace{-0.8cm}
\vspace{0.35cm}
\begin{definition}[Spanning tree]\label{spantree}
A \emph{spanning tree} is a minimal subgraph that contains every vertex.
\end{definition}

\begin{definition}[Activity letters of a signed Tait graph \cite{Coh14}]\label{activitysignedtait}
Given a signed Tait graph $G$, order the edges, and choose a spanning tree $T$ of $G$. We can assign activity letters $L, \ell, D, d, \overline{L}, \overline{\ell}, \overline{D}, \overline{d}$ to weight the ordered edges of the signed Tait graph $G$, given the spanning tree $T$. For positively signed $e\in T$ we give $e$ the weight $L$ if $e$ is the lowest ordered edge that reconnects $T-\{e\}$, and $D$ otherwise. For $e\notin T$, we give the weight $\ell$ if $e$ is the lowest ordered edge in the cycle $T\cup \{e\}$, and $d$ otherwise. For negatively signed $e$, replace every letter with the same, except with a bar, i.e. $L$ is replaced by $\overline{L}$.
\end{definition}

\begin{definition} [Activity word of a graph]
    Given a graph $G$ with edges weighted by activity letters, the \emph{activity word} of $G$ is the product of all the activity letters.
\end{definition}

\begin{example}\label{signedtait1}
    Using our example from figure 8 above, we can find the activity letter specializations for all of the weighted graphs coming from the spanning trees. Then we can find the activity words, $\prod_{e\in G_i}\mu_T$, for each tree, $1\leq i \leq 3$, to be $L^2d$, $LdD$, and $\ell D^2$, respectively.
\begin{center}
    \vspace{5mm}
\includegraphics[width = 80mm]{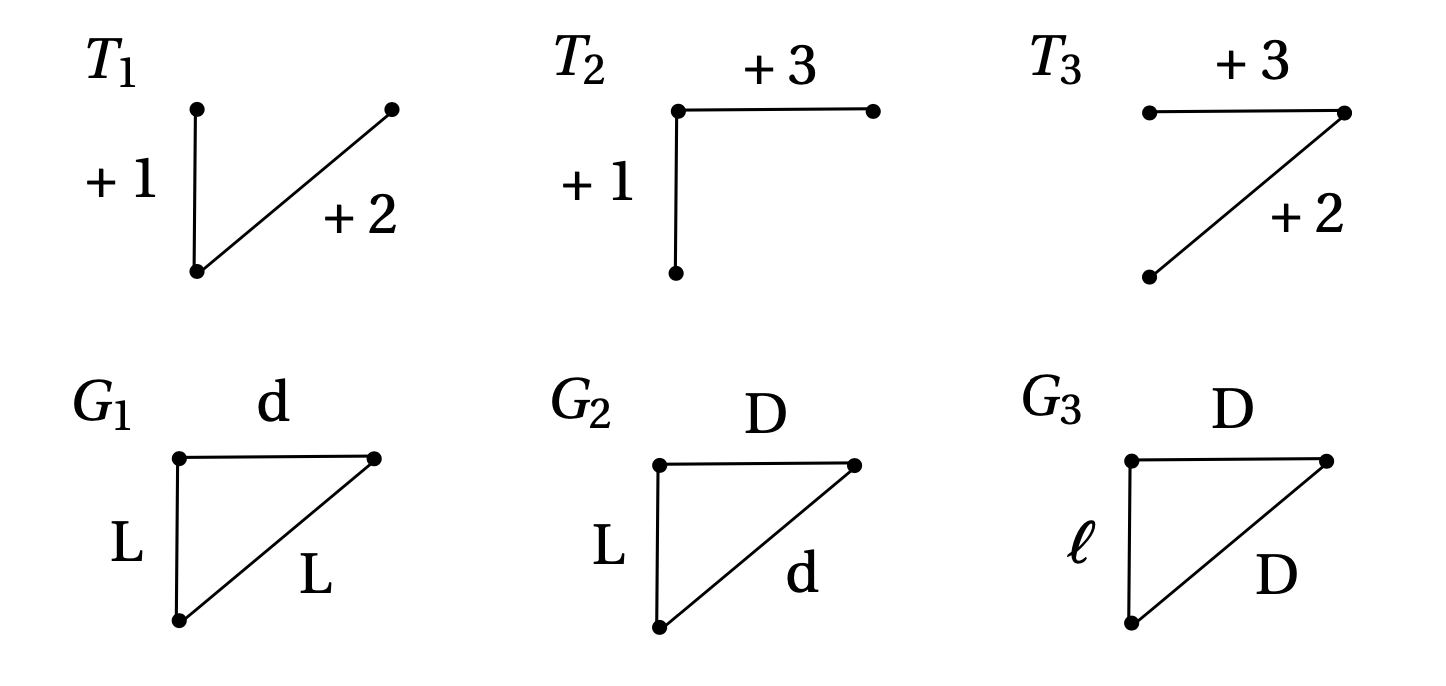}\\
\end{center}
\end{example}

\begin{definition}[Balanced overlaid Tait graph \cite{Coh14}\label{balancedoverlaidtait}] 
Given a checkerboard colored link diagram $D$, we can define a bipartite vertex set $V = V_1 \sqcup W$. Let $V_1$ be the set of crossings, $V_2$ the set of colored faces, $V_3$ the set of uncolored faces. We omit one vertex $v_2$ from $V_2$ and $v_3$ from $V_3$ such that the corresponding faces touch along a strand of $D$ (see Figure \ref{balancedoverlaidtaittrefoil}; the omitted faces touch along a common strand). Now, let $W = (V_2 \cup V_3)-\{v_2,v_3\}$. We say that $v \in V_1$ is adjacent to $w \in W$ if $v$ is a crossing that touches the face $w$. 
\end{definition}
As a matter of convention, let us choose our checkerboard coloring for the balanced overlaid Tait graph so that the outside face is not shaded. Notice that by eliminating two vertices corresponding to faces, by the Euler characteristic, the number of vertices in $V_1$ is the same as the number of vertices in $W$. This will allow perfect matchings to exist.

\begin{figure}[h]
\begin{center}
\includegraphics[width = 60mm]{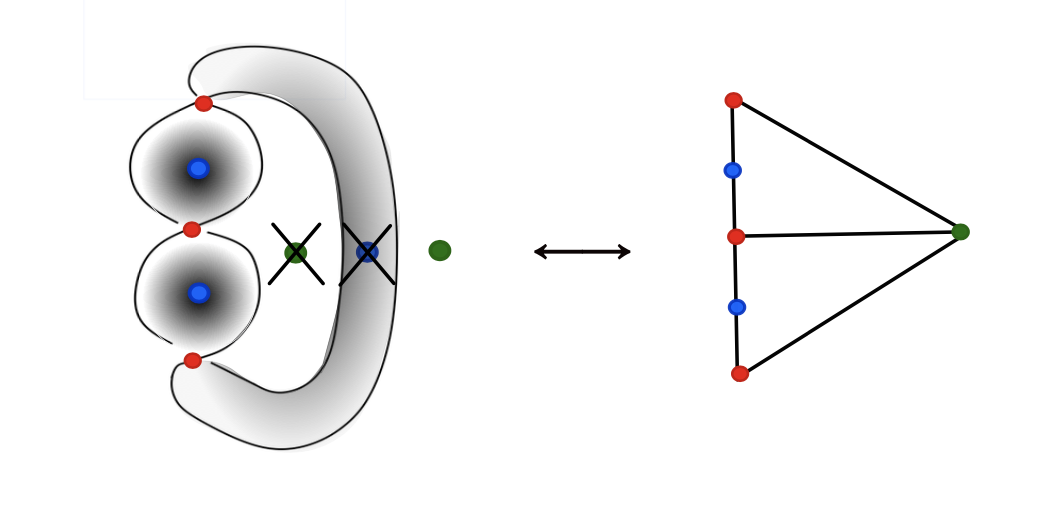}
\caption{balanced overlaid Tait graph of the right-handed trefoil}
\label{balancedoverlaidtaittrefoil}
\end{center}
\end{figure}

\begin{definition}[Activity letters of a balanced overlaid Tait graph \cite{Coh12}]\label{activitybalancedtait}
We can also assign activity letters to weight the edges of the balanced overlaid Tait graph $G$. Start by ordering the vertices in $V_1$. Let $w$ be a vertex corresponding to a shaded or unshaded face. So $w$ is adjacent to ordered vertices. If $w$ is corresponding to a shaded face, weight the edge incident to the lowest ordered adjacent vertex by $L$, and $D$ for every other edge. If $w$ is a vertex corresponding to an unshaded face, then weight the edge incident to the lowest ordered vertex by $\ell$, and $d$ for every other edge.  If some vertex $v\in V_1$ comes from a negative crossing, every weight of edges incident to $v$ picks up a bar i.e. $L$ is replaced by $\overline{L}$. 
\end{definition}

\begin{example}
    Using the balanced overlaid Tait graph of the trefoil given in Figure \ref{balancedoverlaidtaittrefoil}, we can find the activity letters to be:\\
    \begin{center}
\includegraphics[width = 60mm]{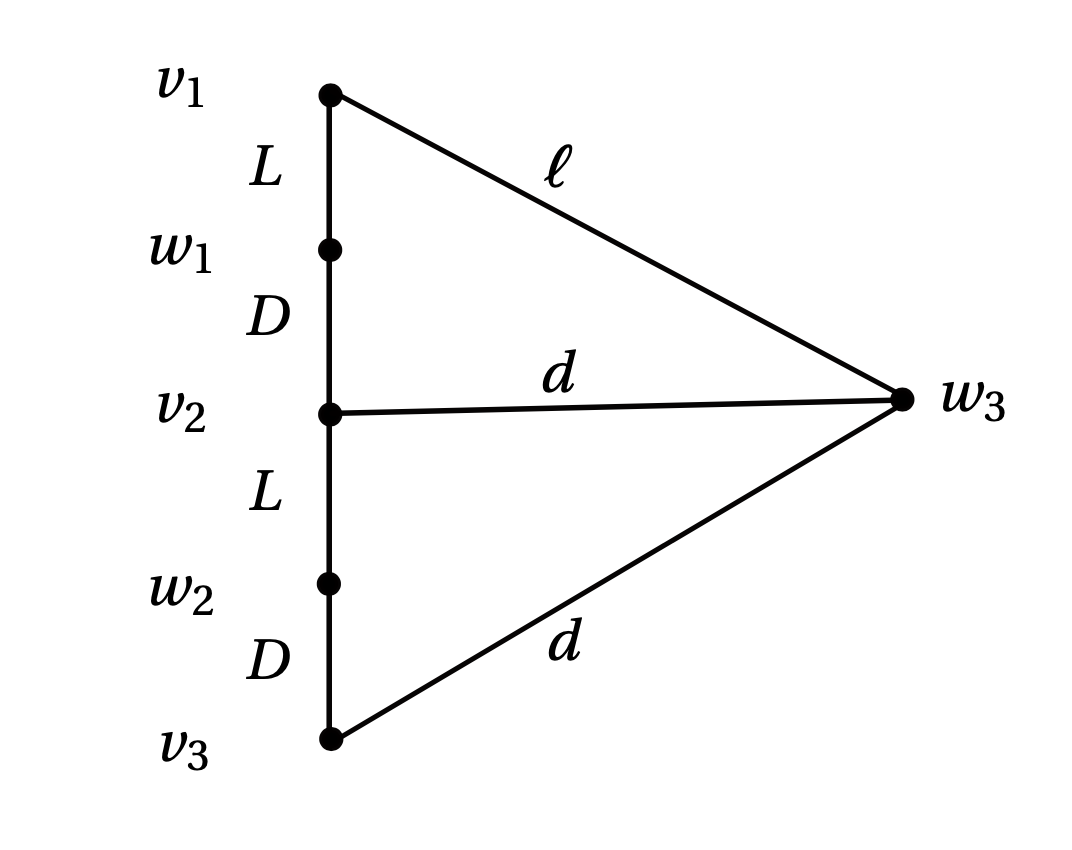}
\end{center}
\end{example}

By above, we have letters in $L, \ell, D, d$ assigned to the signed Tait graph, and the balanced overlaid Tait graph. We can create  polynomials from these letters to recover the Jones polynomial:

\begin{definition}[Activity letter specialization]\label{specialization}
Let $\eta(e)$ be the weight of an edge $e$ (in $L, \ell, D, d, \overline{L}, \overline{\ell}, \overline{D}, \overline{d}$) in either the balanced overlaid Tait graph, or signed Tait graph assigned to some link $L$. The \emph{activity letter specialization $\mu(e)$} is given by
\begin{center}
    
$\mu(e) =\begin{cases} 
     -A^{-3}, & \textrm{if} \ \eta(e) \in \{L, \overline{\ell}\}\\
     -A^3, & \textrm{if} \ \eta(e) \in \{\ell, \overline{L}\} \\
     A, & \textrm{if} \ \eta(e) \in \{D, \overline{d}\} \\
     A^{-1} & \textrm{if} \ \eta(e) \in \{d, \overline{D}\}
   \end{cases}
$
\end{center}

\end{definition}

\begin{theorem}{(Thistlethwaite \cite{Thi86})}\label{thistlewaitestuff}
\[J(L) = (-A^{-3})^{wr(L)}\sum\limits_{T} \prod\limits_{e \in G} \mu_T(e)\]
\\where $G$ is the signed Tait graph of a link $L$, $T$ is a spanning tree of $G$, $\mu_T(e)$ is the weight of $e$ associated to a spanning tree $T$.
\end{theorem}

\begin{example}
    From Example \ref{signedtait1} we can get the activity letters specializations for $G_1$  are $(-A)^{-3}$, $(-A)^{-3}$, and $A^{-1}$, for $G_2$ are $(-A)^{-3}, A^{-1},$ and $A$, and for $G_3$ are $(-A)^{3}, A$ and $A$. This gives that the polynomial associated with the signed Tait graph of the right-handed trefoil will be $(-A)^{-3}((-A^{-3})^2A^{-1}+(-A^{-3})AA^{-1}+(-A^3)A^2)=A^{-4}+A^{-12}-A^{-16}$ which is indeed the Jones polynomial for the right-handed trefoil!
\end{example}

\begin{definition}[Perfect matching, dimer covering]\label{perfectmatching}

A \emph{perfect matching} of a graph $G=(V,E)$ is a subset $P \subset E$ such that every vertex $v\in V$ is incident to exactly one $e\in P$. Equivalently, we can call this a \emph{dimer covering}.
\end{definition}

\begin{example}\label{signedperfectmatchings}
    The perfect matchings of the balanced overlaid Tait graph with activity letter specialization for the right-handed trefoil are shown below. 
\smallskip    
\smallskip   
    \begin{center}
\includegraphics[width = 80mm]{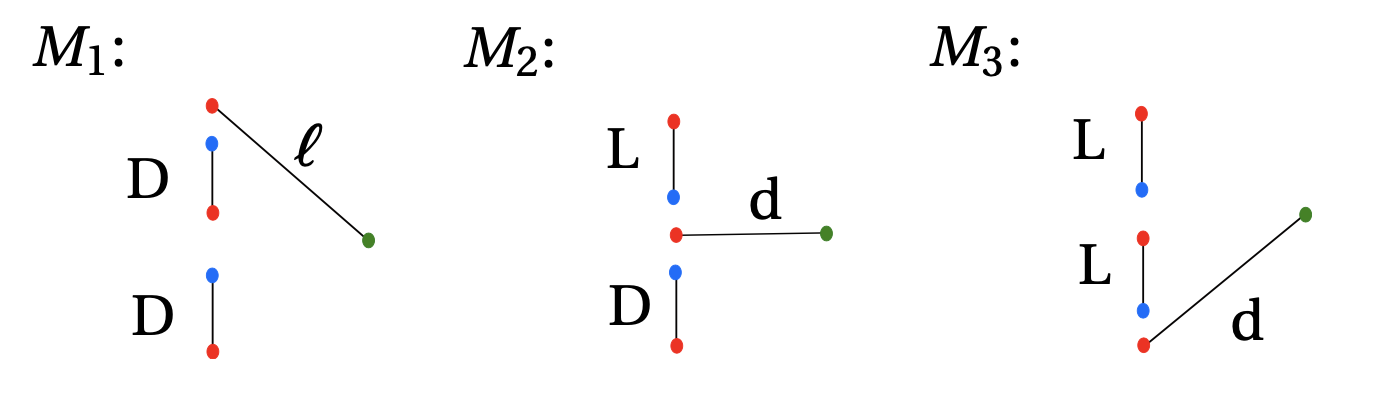}
\end{center}

\end{example}

\begin{definition}[Dimer model]\label{dimermodel}
For an oriented link diagram $L$, let $P$ denote a perfect matching of its balanced overlaid Tait graph (with some ordering of vertices), and $\mu_P(e)$ be the weight of $e$. We say that $L$ admits a \emph{dimer model} if the partition function

\[
    Z(L)=\sum\limits_{P} \prod\limits_{e \in P} \mu_P(e)
\]
\\
equals the bracket polynomial for $L$, $\langle L \rangle$.

\end{definition}

\begin{example}
    The right-handed trefoil admits a dimer model. We can find the activity words with the activity letter specializations corresponding to the each of the perfect matchings of Example \ref{signedperfectmatchings} to be $(-A^3)A^2$, $(-A^{-3})AA^{-1}$, and $(-A^{-3})^2A^{-1}$ for $M_1$, $M_2$, and $M_3$, respectively. Therefore, multiplying by the correction term of $(-A^{-3})^{3}$ gives the polynomial $A^{-4}+A^{-12}-A^{-16}$, which matches the polynomial from the spanning tree method and is thus the Jones polynomial of the right-handed trefoil!
\end{example}

Next we describe how the perfect matching method is indeed a determinant formula. 

\begin{definition}[Modified Adjacency Matrix]\label{modifiedadjacencymat}
We define a modified adjacency matrix $A_G$ of a balanced overlaid Tait graph $G = (V_1\sqcup W, E)$ as $A_G = (a_{v_1,w})$, for $v_1\in V_1, w\in W$, where $A$ has rows labeled by the elements of $V_1$, and columns labeled by elements of $W$. Let $\eta(e_{v_1,w})$ be the weight of the edge connecting $v_1$ and $w$. We have 
\begin{center}
    
$a_{v_{1},w} = \begin{cases} 
     \eta(e_{v_1,w}), & \textrm{if} \ v_1 \ \textrm{adjacent to} \ w \\
     0, & \textrm{otherwise} 
   \end{cases}
$
\end{center}
\end{definition}

\begin{definition}[Kasteleyn Weighting \cite{Coh14}]\label{kasteleynweights}
A \emph{Kasteleyn weighting} is a choice of sign on the edges of a bipartite graph such that the number of negative edges around a face is 
\begin{center}
\begin{align*}
\begin{cases} 
     \textrm{odd}  &  \textrm{if face has length} \ 0 \mod 4 \\
     \textrm{even} & \textrm{if face has length} \ 2 \mod 4
\end{cases}
\end{align*}
\end{center}
\end{definition}
\begin{theorem}[Cohen \cite{Coh14}]\label{cohenkauffmantrick}
Kauffman's trick (see figure below) provides a Kasteleyn weighting for the balanced overlaid Tait graph of an oriented link diagram.\end{theorem}
\begin{figure}[h]
    \centering
    \includegraphics[scale=0.7, width=75mm]{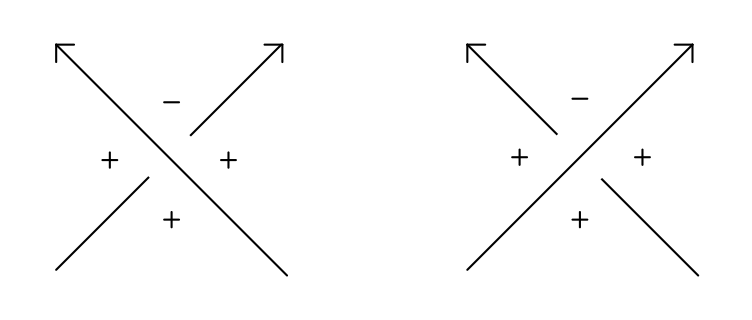}
    \caption{Kauffman's trick for the balanced overlaid Tait graph}
    \label{fig:Kauffman's trick for signing the balanced overlaid Tait graph}
\end{figure}

\begin{theorem}[Cohen \cite{Coh14}]\label{cohenkasteleyn}
If a balanced bipartite graph $G$ is given a Kasteleyn weighting, the determinant of the modified adjacency matrix gives the sum of the perfect matchings of $G$ up to a sign.
\end{theorem}

\begin{example}
    When giving the oriented right-handed trefoil a Kasteleyn weighting, the modified adjacency matrix becomes
    $\begin{bmatrix} L & 0 & \ell \\ -D & L & d \\ 0 & -D & d \end{bmatrix}$. Taking the determinant gives $L^2d+dDL+\ell D^2$, which is the sum of the perfect matchings found in example \ref{signedperfectmatchings}. 
\end{example} 

Theorem \ref{thistlewaitestuff}, \ref{cohenkauffmantrick} and \ref{cohenkasteleyn} tells us that we can always recover a determinant formula for the Jones polynomial when the words from perfect matchings agree with the words from spanning trees.
In particular it allows us to focus solely on the combinatorics of perfect matchings and spanning trees, as this proves to be more concrete than comparing a matrix to spanning trees. Further, Theorem \ref{cohenkasteleyn} tells us that computing the sum of the perfect matchings is a determinant formula and thus can be computed in polynomial time.

\section{$(2,q)$ Torus Knots Have a Dimer Model}\label{torusdimer}

The goal of this section is to show that a class of knots, called $(2,q)$ torus knots, admit a dimer model. Since $(2,q)$ torus knots are a special case of pretzel knots, this is a known result of Cohen \cite{Coh12}. However, we provide a separate proof as the explicit formulas and graph constructions were not given for these class of knots and are needed in future sections.

\begin{definition}[Torus link, torus knot]
    The $(p,q)$ \emph{torus link} is the closure of the braid $\left(\prod_{i=1}^{p-1}\sigma_i\right)^q$ in $B_p$. If $p$ and $q$ are coprime then the closure of $\left(\prod_{i=1}^{p-1}\sigma_i\right)^q$ forms a \emph{torus knot}.
\end{definition}

\begin{remark}
   The sum of all the activity words associated to a knot is the same for its Tait and dual Tait graph. 
\end{remark}

\begin{lemma}\label{taitgraph2qknot}
    The Tait graph of a $(2,q)$ torus link is of the form: 

\begin{center}
 \includegraphics[width = 50mm]{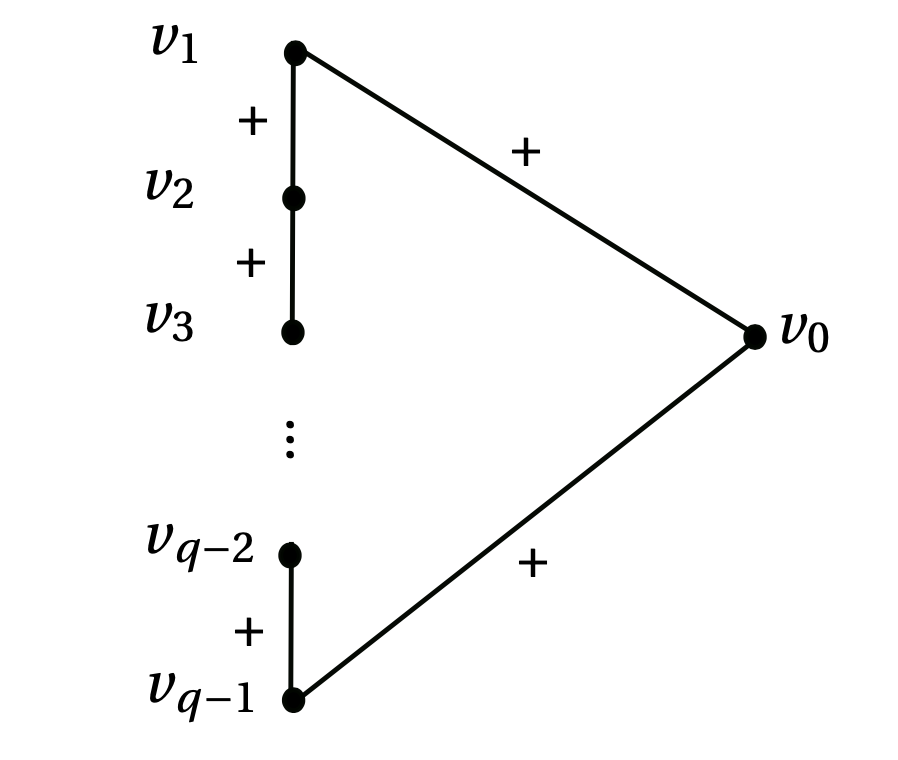}
    \end{center}
    when $q>1$ where the $v_i$ correspond to faces from top to bottom in the column for $1\leq i \leq q-1$ and $v_0$ corresponds to the right-most shaded face. When $q=2$ and $q=1$ the signed Tait graphs are of the form:
    \vspace{0.2cm}
    \begin{center}{\includegraphics[width = 80mm]{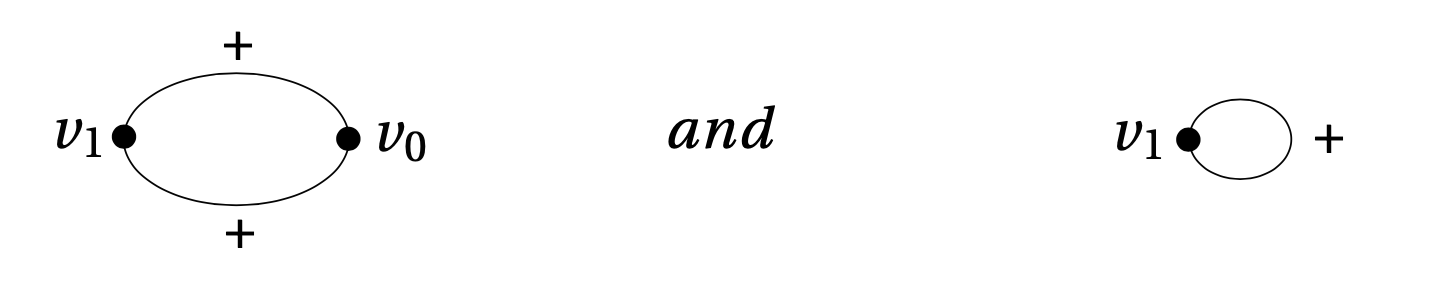}}
    \end{center}
    \noindent respectively, where for $q=2$ $v_1$ corresponds to the face on the left and for $q=1$, $v_1$ corresponds to the single shaded face.     
\end{lemma}

\begin{proof}
    We prove by induction on $q$. Define $\tau_q$ to be the $(2,q)$ torus knot. When $q=1$, then $\tau_1$ is the closure of $\sigma_1$. There is only one face shaded and thus only one vertex. There is exactly one crossing which corresponds to exactly one edge. The crossing is positive, so the edge of the Tait graph is given a positive weight. The Tait graph of $\tau_2$ will have two vertices corresponding to the two faces. There are two crossing giving that these two vertices are connected by two edges. All crossings are positive so the edges are weighted with $+$s. The Tait graph for $\tau_3$ was constructed in Example \ref{signedtaitgraphexample}. 

    Assume that the result is true up to some $q\geq 3$. Consider $\tau_{q+1}$, the closure of the braid $\left(\sigma_1\right)^{q+1}$. Assume that the new crossing for $\tau_{q+1}$ is placed between the first and second crossing. Thus, in a local neighborhood around the first crossing in $\tau_q$, which appears as Figure (\ref{neighborhoodcrossing}) below:

    \begin{figure}[H]
    \includegraphics[width = 80mm]{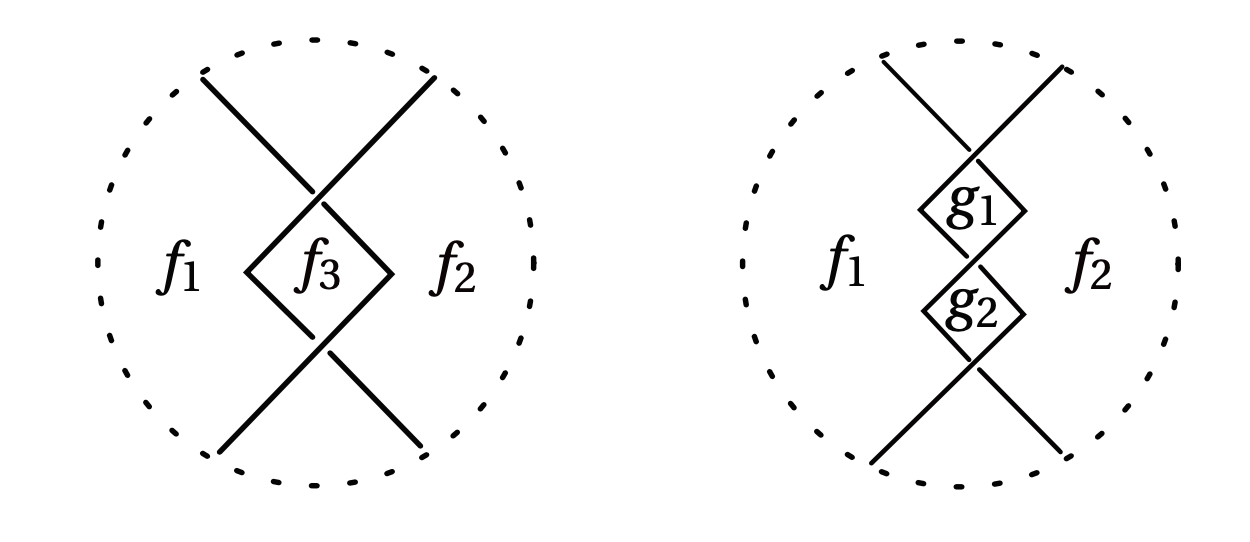}
    \caption{diagram of neighborhood of crossing in $\tau_q$ and $\tau_{q+1}$ respectively.}
    \label{neighborhoodcrossing}
    \end{figure}
    
    \noindent we change the neighborhood to look like the right-hand side of Figure (\ref{neighborhoodcrossing}). We can see that $f_3$ will be a shaded face in $\tau_q$, leaving $f_1$ and $f_2$ unshaded. Therefore, for $\tau_{q+1}$, $g_1$ and $g_2$ are shaded faces. Thus, we have one more vertex for the Tait graph of $\tau_{q+1}$ than $\tau_q$. Further, the face in the knot diagram associated to this new vertex has two crossing associated to the first and third vertex, so this new vertex will have an edge connecting it to the corresponding two vertices. All the crossings are positive, so the edges will have positive weights. Thus the Tait graph is of the desired form, which gives the result. 
\end{proof}

\begin{lemma}\label{balancedtaitgraph2q}
    Let $\tau_q$ be a $(2,q)$ torus knot. When creating the balanced overlaid Tait graph, delete the first two faces from the right, not including the outside face. Then the balanced overlaid Tait graph will be of the form:

    \begin{center}
\includegraphics[width = 50mm]{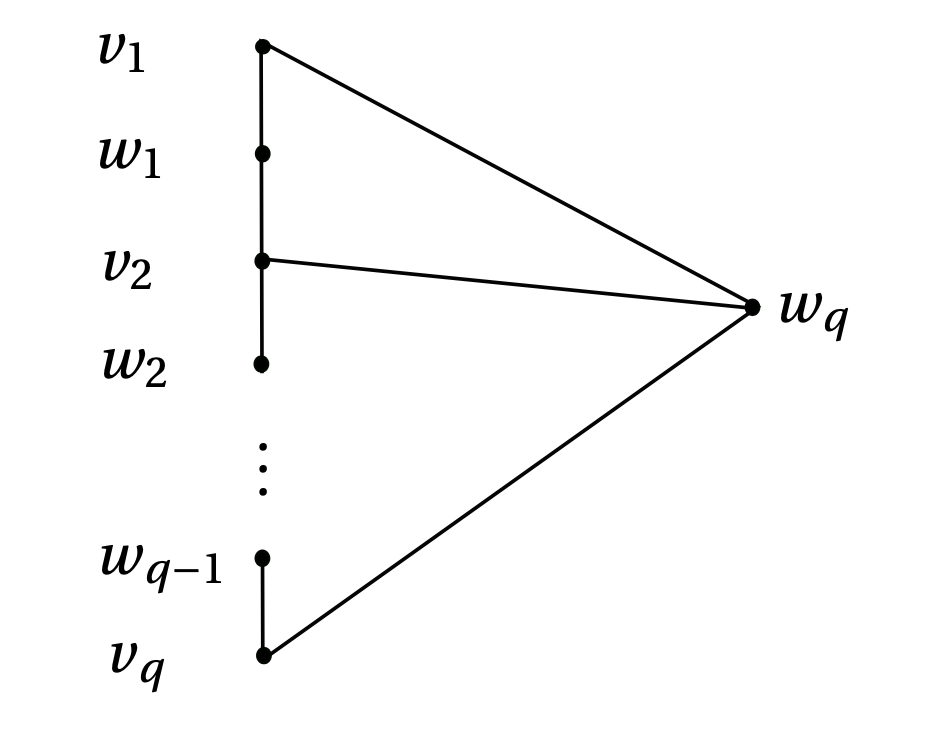}
\end{center}
    and when $q=1$ is of the form:
    \begin{center}
\includegraphics[width = 30mm]{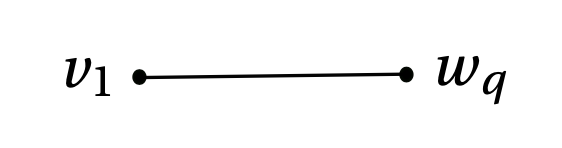}
\end{center}
    \noindent where the $v_i$ correspond to the crossings of the knot for all $1\leq i \leq q$, $w_i$ correspond to the shaded faces of the knot between crossings $v_i$ and $v_{i+1}$ for $1\leq i \leq q-1$, and $w_q$ corresponds to the unshaded crossing of the outside face.
\end{lemma}

\begin{proof}
    We will prove this lemma by induction on $q$. First assume $q=1$. Then then $\tau_1$ is the closure of the braid $\sigma_1$. Deleting the two interior faces, we can easily see that the balanced overlaid Tait graph will be of the desired form. Next, consider the balanced overlaid Tait graph of $\tau_2$. There are two crossings in $\tau_2$ which will correspond to vertices $v_1$ and $v_2$. There are four faces of $\tau_2$. We delete the innermost unshaded face and the leftmost shaded face, leaving one shaded face between the two crossings and the outside face. Then $v_1$ and $v_2$ are both connected to the two vertices corresponding to the faces, giving the desired graph. Notice that the balanced overlaid Tait graph of $\tau_3$ was made in Figure (\ref{balancedoverlaidtaittrefoil}). 

    Assume that the result is true up to some $q\geq 3$. Then $\tau_{q+1}$ is the closure of the braid $\sigma_1^{q+1}$. This link looks identical to $\tau_q$ except we have an additional crossing. Assume the additional crossing is placed between the 1st and 2nd crossing. That is, in a local neighborhood around the first and second crossing in $\tau_q$, as shown on the left-hand side of Figure (\ref{neighborhoodcrossing}), we change the neighborhood to look like the right-hand side of Figure (\ref{neighborhoodcrossing}).
    
    By induction, we know in between the two crossings in the first figure must be a shaded face. Therefore, faces $f_1$ and $f_2$ will be unshaded, making $g_1$ and $g_2$ shaded faces. Therefore, the vertices in the balanced overlaid Tait graph coming from the new faces will both correspond to shaded faces. Further, the new vertices still connect to the outside face, which will give that the balanced overlaid Tait graph of $\tau_{q+1}$ will be of the desired form. 
\end{proof}

\begin{proposition}\label{activitywordtorus1}
Let $P_q$ be the the sum of all the activity words associated to the perfect matchings of a balanced overlaid Tait graph of a $(2,q)$ torus knot. We have that
\[P_q = lD^{q-1}+\sum_{i=1}^{q-1} dL^{i}D^{q-1-i} \].
\end{proposition}

\begin{proof}
    Recall from Lemma (\ref{balancedtaitgraph2q}) the form of the balanced overlaid Tait graph of a $(2,q)$ torus knot. We then can find the activity letters to be:

\begin{center}
\includegraphics[width = 50mm]{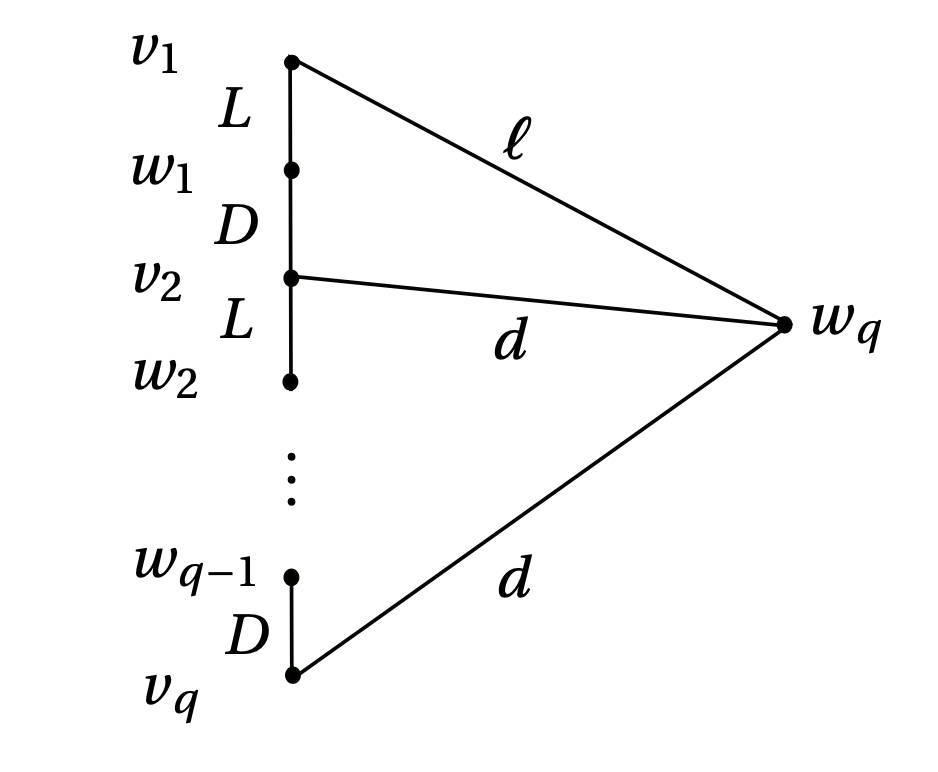}
\end{center}

    A perfect matching is completely determined by the choice of edge for $w_q$, of which there are $q$ options. Choosing the top edge labelled by an $\ell$ we get that the rest of the edges chosen along the column will be labelled $D$, due to the alternating structure of $L$ and $D$. There are $q-1$ such edges. Thus this perfect matching gives activity word $\ell D^{q-1}$.  

    Now choose any other edge connected to $w_q$, of which there are $q-1$ options. This edge is labelled by $d$. Say this edge is connected to some vertex $v_i$ where $2\leq i \leq q$. Above $v_i$, due to the alternating nature of $L$ and $D$, the edges for the perfect matchings must be labelled by $L$, and there will be $i$ of them. Similarly the $q-1-i$ edges below $v_i$ will be labelled $D$. Therefore the activity word for this perfect matching is $dL^iD^{q-i}$. Therefore, we obtain
    \begin{align*}
        P_q = lD^{q-1}+\sum_{i=1}^{q-1} dL^{i}D^{q-1-i}
    \end{align*}
    as we desired.
\end{proof}

\begin{proposition}
    Let $S_q$ be the sum of all the activity words associated to the spanning trees of a Tait graph of a $(2,q)$ torus knot. Then $S_q=P_q$ and thus $(2,q)$ torus knots admit a dimer model.
\end{proposition}

\begin{proof}
    Recall from Lemma \ref{taitgraph2qknot} the form of the Tait graph of a $(2,q)$ torus knot. Removing any edge from the graph results in a spanning tree. Label the edge connecting $v_0$ and $v_1$, $e_0$. Then in a counterclockwise fashion, label the rest of the edges. Call $T_k$ the spanning tree that removes edge $e_k$. If $k=0$, then that edge is the lowest edge in the cycle $T_0 \cup \{e_0\}$, so it is given a weighting of $\ell$. If $k\neq 0$ then $e_k$ is not the lowest edge in the cycle and will be given the weight $d$. 

    For any $k$ consider $e_j$ where $j\neq k$. If $j<k$ then $e_j$ is the lowest edge that reconnects $T_k-\{e_j\}$. If $j>k$ then $e_k$ is the lowest edge that reconnects $T_k-\{e_j\}$. Thus when $j<k$, $e_j$ is labelled $L$ and if $j>k$, $e_j$ is labelled $D$, which gives that $P_q=S_q$.  
\end{proof}

\begin{corollary}\label{dimerfortorusknots}
Every $(p,2)$ torus knot admits a dimer model.
\end{corollary}

\begin{proof}
It is known that every $(p,q)$-torus knot is isotopic to a $(q,p)$-torus knot. This implies that the $(p,2)$-torus knot is isotopic to a $(2,q)$ torus knot which can be obtained as the closure of the braid $\sigma_1^q$. By Cohen \cite{Coh14} (and our proof above), the $(2,p)$-torus knots admit a dimer model which implies that $(p,2)$-torus knots admit a dimer model. 

\end{proof}


\section{A Class of Homogeneous Closed Braids Have a Dimer Model}\label{hombraiddimer}

In this section, we show that links obtained as the closure of homogeneous braids admit a dimer model. This extends the work in \cite{Coh12} to a new family of links.

\begin{theorem}\label{dimerforbraidclosure2}
Let $L$ the closure of the word $\sigma_1^{m_1}\sigma_2^{m_2}\dots\sigma_{n-1}^{m_{n-1}}$ where $m_1,...,m_{n-1}\in\mathbb{N}$ or $-m_1,...,-m_{n-1}\in \mathbb{N}$. Then $L$ admits a dimer model.
\end{theorem}
\begin{proof}
We first prove the case when $n=3$. 

\begin{lemma}\label{dimerforbraidclosure1}
Let $\sigma_1^{m_1}\sigma_2^{m_2}\in B_3$ be a braid word where either $m_1,m_2 \in \mathbb{N}$ or $-m_1,-m_2 \in\mathbb{N}$ and let $L$ be the closure of this braid. Then $L$ admits a dimer model. 
\end{lemma}

\begin{proof} 
The proof for $m_1, m_2 \in \mathbb{N}$ is similar to $-m_1, -m_2 \in \mathbb{N}$, so we just just prove for $m_1,m_2 \in \mathbb{N}$. We will order the crossings in the balanced overlaid Tait graph from top to bottom (or equivalently, by the order they show up in the word). We have the balanced overlaid Tait graph has two components when we delete the two rightmost faces sitting between the first and second strand, and second and third strand. Hence, we can determine the perfect matchings component-wise. This is straightforward though, as each component is isomorphic to the balanced overlaid Tait graph of a $(2,q)$ torus knot. These knots are a special case of pretzel knots, which is known to admit a dimer model by \cite{Coh12}. So, now we must show that the spanning trees of the Tait graph give us the activity words (coming from perfect matchings) of two $(2,q)$ torus knots whose graphs are isomorphic to the components of our graph. 
\\ The Tait graph of $L$ is two Tait graphs of $(2,q)$ torus knots glued together at a common vertex. We will denote $G_1$ the triangle subgraph with $|m_1|$ edges, and $G_2$ the subgraph with two vertices and $|m_2|$ edges. We observe that $G_1$ is a Tait graph of the closure of the word $\sigma_1^{m_1}\in B_2$. Further, $G_2$ is a Tait graph (arising from the dual checkerboard coloring) of the closure of the word $\sigma_1^{m_2}\in B_2$. Now, we observe that the letter assignment in the spanning tree contributed by $G_1$ is independent of the assignment in $G_2$, and vice versa. Hence, we take all products words of the spanning trees of $G_2$ with those of $G_1$ which is equivalent to taking all products of perfect matchings on a disconnected bipartite graph. 
\end{proof}

 With only minor modification the proofs of $m_1,...,m_{n-1}\in \mathbb{N}$ and $-m_1,...,-m_{n-1}\in \mathbb{N}$ are the same, so we will only prove for $m_1,...,m_{n-1}\in \mathbb{N}$. We argue similarly as in Lemma \ref{dimerforbraidclosure1} above. If $n$ is odd, the face coming from the closure of the right-most strand is colored; if $n$ even, it is uncolored. We make the same observation that the balanced overlaid Tait graph of $L$ has subgraphs isomorphic to those of $(2,q)$ torus knots, where we delete the two rightmost faces sitting between the first and second strand and the second and third strand. To argue identically, we must show that the edges connecting these subgraphs don't show up in any perfect matching, so we have essentially disjoint subgraphs. If we pick an edge connecting to the right-most vertex coming from the inner-most face, this determines every edge in the $\sigma_{n-1}$ column. In particular, no other edge touching this column shows up in a perfect matching. Hence, we can factor out the word from the $\sigma_{n-1}$ column, and reduce to the case of some braid with $n-1$ strands, which completes the induction.
\end{proof}


\begin{remark}\label{general_adjacency_matrix}
Although it is not included, one can write a closed formula of the adjacency matrix for any given braid word.
\end{remark}


\section{Kauffman Polynomials Arising from Balanced Overlaid Tait Graphs}\label{Kauffman Polynomials}

Although we initially studied the same invariant as the author in \cite{Coh12}, their methodology should be explored in other settings. One of the key components in the argument is the formalism of the Jones polynomial in terms of a skein relation. For this reason, we study the Kauffman polynomial which can also be defined in terms of a skein relation.

\begin{definition}[Kauffman Polynomial \cite{Kau90}]\label{kauffmanpolynomial}
The \emph{Kauffman polynomial} $F(L)$ of a link diagram $L$ is defined by 
\[F(L) = a^{-wr(L)}K(L)\]
where the polynomial $K(L)$ of a link $L$ is defined by the skein relation (see Fig.\ref{crossingsdef}):
\[L_+ + L_- = z(L_a + L_b)\]
Further, if we resolve a positive kink, $L_{pos}$, in our link, by Reidemeister 1, we have $K(L_{pos}) = a^{-1}K(L)$. Similarly for a negative kink, we have $K(L_{neg}) = aK(L)$. Additionally, $K(O) = 1$ where $O$ is the unknot. Finally, we have $K(L \sqcup O) = (\frac{a+a^{-1}}{z} -1)K(L)$. 
\\
\end{definition}

\begin{proposition}\label{kauffmanpolyntorusknot} Let $P_q$ be as in Proposition \ref{activitywordtorus1} for $q\geq 2$. Then
\[
    K(2,q) = P_q - \sum_{i=0}^{q-2} z^{q-2-i}\cdot K(2,i)\]
where we specialize via \[L = a^{-1} = l^{-1}\] \[D=d=z\]
\end{proposition}
\begin{proof} 
Let $(2,q)$ denote the $(2,q)$ torus knot. Then \[K(2,q) = za^{q-1} + zK(2,q-1) - K(2, q-2)\] by examining the skein relations for $K(L)$.

We have that \[K(2,0) =\frac{a+a^{-1}}{z}-1\coloneqq P_0\] 
\[K(2,1) = a^{-1} \coloneqq P_1\] and we check the base case for $K(2,2)$.
We have \[K(2,2) = za + za^{-1} - K(2,0) = Ld + lD - K(2,0)\]
For our inductive step, we have \[K(2,q) = za^{q-1} + zK(2,q-1) - K(2,q-2)\]
By Proposition \ref{activitywordtorus1}, we note that \[K(2,q) = za^{q-1} +zP_{q-1} - \sum_{i=0}^{q-2} z^{q-2-i}\cdot K(2,i) = P_q - \sum_{i=0}^{q-2} z^{q-2-i}\cdot K(2,i)\]
where $P_q = za^{q-1} + zP_{q-1}$ by the skein relations for the bracket polynomial.
\end{proof}
\vspace{10mm}
\hspace{-4.2mm}Now, we introduce a recursive relation defined as 
\[g_0 = 1\]
\[g_1 = z\]
\[g_n = zg_{n-1} - g_{n-2}\]

\begin{remark}
It is worth noting that this relation is reminiscent of the recursive Chebyshev polynomials one would expect from the colored Jones polynomial. See \cite{KM91} by Kirby and Melvin.
\end{remark}

\begin{proposition}\label{recursiverelation}
For $n\geq 2$, 
 \[g_n = z^n - \sum_{i=0}^{n-2} z^{i}\cdot g_{n-2-i}\]
\end{proposition}
\begin{proof}
We proceed by induction. For $n=2$: 
\[g_2 = zg_1 - g_0 = z\cdot z - 1 = z^2 - z^0g_0\]
For $n=3$, 
\[g_3=zg_2-g_1=z^3-2z=z^3z^0g_1-z^1g_0\]
For the inductive step, by the recursion relations, we have that 
\[g_n = zg_{n-1} - g_{n-2} = z(z^{n-1} - g_{n-3} - zg_{n-4} - \dots - z^{n-4}g_1 - z^{n-3})-(z^{n-2} - g_{n-4} - zg_{n-5} - \dots - z^{n-5}g_1 - z^{n-4}) \]
We substituted for $g_{n-1}$ and $g_{n-2}$ by our inductive hypothesis. Then, by rearranging our terms, we get
\[g_n = z^n - (zg_{n-3} - g_{n-4}) - z(zg_{n-4} - g_{n-5}) - \dots -  z^{n-4}(zg_1 - g_0) - z^{n-3}g_1 - z^{n-2} \]
Then, we can substitute again with recursion relations, completing the induction:
\[g_n = z^n - g_{n-2} - zg_{n-3} - \dots - z^{n-2} = z^n - \sum_{i=0}^{n-2} z^{i}\cdot g_{n-2-i}\]
\end{proof}

\begin{theorem}\label{kauffmanpolynrecursive}
    For $n\geq 2$, \[K(2,q) = P_q - \sum_{i=0}^{q-2} P_i\cdot g_{q-2-i}\]
\end{theorem}
\begin{proof}
Straightforward induction, appealing to Proposition \ref{kauffmanpolyntorusknot} and \ref{recursiverelation}.
\end{proof}

\section{Further Directions}\label{furtherdirections}
Because of the explicit nature in the computations of this work, there are many natural extensions of our results in which we will consider in the future. For example, by the work of Thurston \cite{Thu82}, it is well-known that every knot is either a torus knot, hyperbolic knot, or a satellite knot. Since torus knots have a straightforward description in terms of the braid group, we intend to use its weighted adjacency matrix to prove the following conjecture:
\begin{conjecture}
Every $(p,q)$ torus knot admits a dimer model
\end{conjecture}
Additionally, we are interested in exploring other families of knots in which we can extend our results.

Another possible avenue we will consider is extending our results for other polynomial invariants such as in Section \ref{Kauffman Polynomials}. Including the Jones polynomial, many invariants from the field of quantum topology can be derived from combinatorial relations on the knot diagram known as skein relations such as in Figure \ref{crossingsdef}. Because many of our proofs rely on the recursive nature of the skein relations, it seems reasonable to believe that there are other modifications to our graph which will recover different polynomial invariants. In particular, it is known that the family of knot invariants known as the $N$-colored Jones polynomial can be obtained from taking a sum of bracket polynomials on different cablings of a knot diagram. In this sense, every $N$-colored Jones polynomial corresponds to a finite number of link diagrams in which a dimer model, possibly using the same construction, may exist. 

Lastly, we consider exploring known operations of knots and their relations towards dimer models. As mentioned in Remark \ref{general_adjacency_matrix}, it is possible to write out the  adjacency matrix for any closed braid. There are also various operations that could be performed on knot diagrams such as the connect sum, forgetful homomorphism, and satellite operations. In future works, we would like to explore if there are certain properties of the adjacency matrix which are well behaved under an operation.

\bibliographystyle{amsplain}

\end{document}